\documentclass[10pt]{article}
\usepackage{amsmath,amssymb}
\usepackage{graphicx}
\usepackage{appendix}
\usepackage[font={footnotesize}]{caption}
\usepackage[font={footnotesize}]{subcaption}
\usepackage{theorem}
\usepackage{esint}
\usepackage{color}
\usepackage{booktabs}
\usepackage{multirow}
\usepackage{lineno}

\usepackage{latexsym}
\usepackage{epsfig}
\usepackage{slashbox}
\usepackage{amsmath}
\usepackage{amssymb}
\usepackage{amsfonts}
\usepackage{dsfont}
\usepackage{amssymb}
\usepackage{stmaryrd}
\usepackage{multirow}
\usepackage{epstopdf}
\usepackage{amsbsy}

\usepackage{bm}
\usepackage{enumerate}

\numberwithin{equation}{section} {\theorembodyfont{\upshape}
\newtheorem{remark}{\textbf{\emph{Remark}}}[section]
\newtheorem{mytheorem}{Theorem}[section]

\theoremstyle{plain}\newtheorem{myprop}[mytheorem]{Proposition}
\theoremstyle{plain} 
\theoremstyle{plain}
\numberwithin{equation}{section}
{\theorembodyfont{\upshape}\theoremstyle{plain}}
\newenvironment{proof}{\noindent{\em Proof.}}{\quad \hfill{\footnotesize $\Box$}\vspace{1ex}}

\newcommand{\myend}{\hfill{\quad \hfill{\footnotesize $\Box$}\vspace{0.1ex}}}

\newcommand{\RNum}[1]{\uppercase\expandafter{\romannumeral #1\relax}}

\usepackage[left=3cm,right=2.5cm]{geometry}

\begin{document}

\title{Condensed Generalized Finite Element Method (CGFEM)}
\author{Qinghui Zhang \thanks{Corresponding author, Guangdong Province Key Laboratory of Computational Science and School of Data and Computer Science, Sun Yat-Sen University, Guangzhou, P. R. China. Email address: zhangqh6@mail.sysu.edu.cn. This research was partially supported by the Natural Science Foundation of China under grants 11471343 and 11628104, and Guangdong Provincial Natural Science Foundation of China under grant 2015A030306016.}
\and Cu Cui \thanks{School of Mathematics, Sun Yat-Sen University, Guangzhou, P. R. China. Email address: cuic3@mail2.sysu.edu.cn.}
}

\date{}

\maketitle

\begin{abstract}
Generalized or extended finite element methods (GFEM/XFEM) are in general badly conditioned and have numerous additional degrees of freedom (DOF) compared with the FEM because of introduction of enriched functions. In this paper, we develop an approach to establish a subspace of a conventional GFEM/XFEM approximation space using partition of unity (PU) techniques and local least square procedures. The proposed GFEM is referred to as condensed GFEM (CGFEM), which (i) possesses as many DOFs as the preliminary FEM, (ii) enjoys similar approximation properties with the GFEM/XFEM, and (iii) is well-conditioned in a sense that its conditioning is of the same order as that of the FEM. The fundamental approximation properties of CGFEM is proven mathematically. The CGFEM is applied to a problem of high order polynomial approximations and a Poisson crack problem; optimal convergence orders of the former are proven rigorously. The numerical experiments and comparisons with the conventional GFEM/XFEM and FEM are made to verify the theory and effectiveness of CGFEM.
\end{abstract}


\section{Introduction}

Generalized or extended finite element methods (GFEM/XFEM) augment the standard finite element methods (FEM) with special functions that locally
mimic unknown solutions of underlying problems. Meshes in the GFEM/XFEM are typically simple, fixed, and independent of non-smooth features of problems, and thus the heavy burden of re-meshing and mesh refinement in the FEM is alleviated essentially. The GFEM/XFEM has been extensively applied to a wide range of engineering problems, e.g., crack propagation, interface, multiphase flows, fluid-structure interaction, large deformation. We refer to review articles \cite{BaBanOs,BaBanOs1,BelGra,Efe,FrBel} and references therein for various aspects of GFEM/XFEM. The GFEM/XFEM can be viewed as an instance of the partition of unity methods (PUM) \cite{BaMe,MelBa1}. In the rest of paper, we will use the term GFEM instead of GFEM/XFEM.

It was realized early that the GFEM is generally badly conditioned. This is mainly caused by almost linear dependence between the FE shape functions and added special functions. The bad conditioning in the GFEM may cause disastrous round-off errors in elimination methods or slow convergence of iterative schemes in solving underlying linear systems. Many interesting ideas have been proposed to improve the conditioning of GFEM, such as, changing the standard linear FE PU functions to so-called flat-top PU functions \cite{GriSch,HLi,ZhaBaBan}, preconditioning the stiffness matrices or orthogonalization \cite{AgaBor,AgaCha,BeMoBu,MenBor,LanMak,Sch}, and correcting the enrichments by interpolant \cite{BaBan1,BabBanKen,ZhaBaBan,ZhaBaBan2,ZhaBaBan1,GuDuBaBan,SanDu}.

In addition to the bad conditioning, the introduction of special functions in the GFEM causes complex approximation spaces that are composed of the FE spaces and enriched spaces. Correspondingly, the stiffness matrix of GFEM consists of the FE part and enriched part also. Such a structure makes many standard operations in the FEM, for instance, time integrations, mass lumping, complicated in the GFEM, see \cite{FrZil,WenTian,MenRet} for instance. Therefore, many studies have been devoted to reducing DOFs of the GFEM. It is quite natural to believe that the reduction of DOFs can also decrease the almost linear dependence of shape functions, and thus improve the conditioning of GFEM. A DOF-gathering GFEM was proposed in \cite{LabJul,NicRenCha,Zhang0} for the crack problems, where numerous singular enriched DOFs are gathered to several using cut-off functions. The optimal approximations of DOF-gathering GFEM were attained \cite{LabJul,NicRenCha}, and the well conditioning was addressed in \cite{Zhang0}. The PU technique can also be a powerful approach to reduce the DOFs of GFEM. We sketch this idea briefly. For each PU function, the local enriched functions are re-combined according to interpolations or least square approximations such that the new enriched functions are characterized by the function values at nodes. Then, the global shape functions are established by collecting the local enriched functions based on the functions values at nodes. The PU ideas to reduce the DOFs were applied to develop the intrinsic GFEM (PUM) \cite{FrBelInt,FrBelIntPU,Fries}, the extra-DOF-free GFEM \cite{Tian}, and so on. We mention that these ideas were also employed in PUM-RBF methods \cite{MolFer,Wen1}, meshless methods \cite{FleChu,NguRab}, and so on, to reduce DOFs or improve the conditioning. The shape functions of the intrinsic GFEM are constructed using moving least-square procedures, which cause the well-known numerical integration issue as in the meshless methods \cite{BBOZ}. The extra-DOF-free GFEM employs a selectively interpolating least-square approach, which poses a strong constraint when the linear functions serve as the enrichments. Moreover, there is no theoretical analysis on all the above-mentioned DOF-reduced GFEMs. Theoretical understanding is important to design more efficient DOF-reduced GFEMs.

In this paper, we propose a condensed GFEM (CGFEM) where the shape functions of conventional GFEM are condensed using the PU techniques. The approximation space of CGFEM is a subspace of the preliminary GFEM space. The CGFEM possesses the same number of DOFs as that of the FEM and enjoy the approximation advantages of GFEM at the same time. The local approximation functions are constructed by the least-square scheme instead of the moving least-square method so that the difficulty of numerical integrations in the meshless methods does not exist in the CGFEM. The fundamental approximation properties of CGFEM are proven mathematically, which are almost same as those of the conventional GFEM. We find out from the approximation properties that the regularity of shape functions plays an important role in the CGFEM. The scaled condition number of CGFEM is numerically shown to be of the same order as that of the FEM. The CGFEM is applied to a problem of high order polynomial approximations and a Poisson crack problem successfully. The regularity results and optimal convergence orders are proven rigorously for the high order polynomial approximations. The numerical experiments and comparisons with the conventional GFEM and FEM are made to verify the theory and effectiveness of CGFEM. In addition, as in the FEM, the shape functions of CGFEM are characterized by nodes in that a shape function is associated with a node. Therefore, many standard operations in the FEM, such as the time integrations and mass lumping, can be applied to the CGFEM more directly. This is another potential advantage of CGFEM over the conventional GFEM, which will be investigated in forthcoming studies.

The paper is organized as follows. The model problem is described in Section 2, where the conventional GFEM and stable GFEM are reviewed as well. The new CGFEM is proposed in Section 3, and the fundamental approximation properties of CGFEM are proven. In Section 4, the CGFEM is applied to the high order polynomial approximations and a Poisson crack problem, the regularities and optimal convergence orders are proven for the high order polynomial approximations. The numerical experiments and comparisons with the conventional GFEM, SGFEM, and FEM are made in Section 5. The concluding remarks are presented in Section 6.

\section{Model problem, GFEM, and SGFEM}\label{sec2}

For a domain $D$ in ${\mathbb R}^d$, an integer $m$, and $1 \leq q \leq \infty$,  we denote the usual Sobolev spaces by $W^{m,q}(D)$ with
norm $\|\cdot\|_{W^{m,q}(D)}$ and semi-norm $|\cdot|_{W^{m,q}(D)}$. The space $W^{m,q}(D)$ will be represented by $H^m(D)$ in case $q=2$ and $L^q(D)$ when $m=0$, respectively.

Let $\Omega\subset \mathbb R^d$ be a bounded, simply connected, convex domain with piecewise smooth boundary $\partial \Omega$. We consider an elliptic variational problem as follows:
\begin{eqnarray}
\mbox{Find } u \in {\cal E}(\Omega)\subset H^1(\Omega) \mbox{ such that } B(u,v) = L(v), \quad \forall \ v \in {\cal E}(\Omega), \label{varprob}
\end{eqnarray}
where ${\cal E}(\Omega)$ is an energy space defined by
\[{\cal E}(\Omega):=\{v\in H^1(\Omega): B(v,v)< \infty\}\]
with an energy norm $\|v\|_{{\cal E}(\Omega)}:=\sqrt{B(v,v)}$, $B(\cdot,\cdot)$ is the usual energy inner product satisfying
\[
C_1|v|_{H^1(\Omega)}^2 \le B(v,v) \le C_2 |v|_{H^1(\Omega)}^2, \quad  \forall \ v \in {\cal E}(\Omega)
\]
where $C_1,\, C_2$ are fixed positive constants, and $L(\cdot)$ is a bounded linear functional on $H^1(\Omega)$. (\ref{varprob}) could be the general variational formulations for many elliptical problems, such as Poisson equations, elasticity equations. To highlight the main ideas, we consider the Neumann boundary conditions in this paper. Impositions of the essential boundary conditions will be commented in the next section.

Let ${\cal S}_h\subset{\cal E}(\Omega)$ be a finitely dimensional approximate subspace, and a discretized variational problem to (\ref{varprob}) based on ${\cal S}_h$ is posed by
\begin{eqnarray}
\mbox{Find } u_h \in {\cal S}_h \mbox{ such that } B(u_h,v_h) = L(v_h), \quad \forall \ v_h \in {\cal S}_h. \label{dvarprob}
\end{eqnarray}
It is easy to know from the Cea's Lemma that
\begin{equation}\label{Cea}
\|u-u_h\|_{{\cal E}(\Omega)}\leq \min_{v_h\in {\cal S}_h}\|u-v_h\|_{{\cal E}(\Omega)}.
\end{equation}
The result (\ref{Cea}) points out that an approximation space ${\cal S}_h$ with good approximation properties will produce a good approximation solution $u_h$ from (\ref{dvarprob}).

\medskip

\noindent\textbf{GFEM, and Flat-Top PU GFEM:}

\medskip

To describe the GFEM, we consider a standard FE mesh on $\Omega$, for a given mesh-size parameter $0<h<1$. Let $\{e_s:\,s\in E_h\}$ be the set of elements with an index set $E_h$, and every $e_s$ is closed. The set of nodes is denoted by $\{x_i \, : \,i \in I_h\}$ with the index set of nodes, $I_h$. We assume that the mesh is quasi-uniform and independent of non-smoothness in exact solutions. Let $N_i$ be the standard FE hat function associated with the nodes $x_i$, $i \in I_h$, with closed supports $\omega_i$. The closed sets, $\omega_i, i\in I_h$, are called patches. Since the mesh is quasi-uniform, we have
\begin{equation}
\quad \|N_i\|_{L^\infty(\Omega)} \le C_1,\quad \|\nabla N_i\|_{L^\infty(\Omega)} \le C_2 h^{-1} \label{Shacond}
\end{equation}
for all $i \in I_h$, where $C_1,\, C_2$ are generic constants independent of $i$ and $h$. It is well known that $ N_i,\, i \in I_h$ form a PU, i.e.,
\[
\sum_{i \in I_h} N_i(x) \equiv 1, \quad \forall \  \ x\in \Omega.
\]
It is well known that the standard FEM produces poor accuracy when the exact solutions are not smooth, e.g., singular or discontinuous.

The GFEM \cite{BaBanOs,BaBanOs1} is a Galerkin method with special approximation spaces ${\cal S}_h$, which are obtained by augmenting the FE spaces with the special functions that locally ``mimic'' unknown solutions of underlying problems. The GFEM can be viewed as a particular instance of partition of unity methods (PUM) \cite{BaMe,MelBa1}, where local approximation functions are coupled using the partition of unity (PU) functions. On each patch $\omega_i$, we consider an $n_i$-dimensional local approximation subspace
\[
V_i=\mbox{span}\{\xi_i^{[j]}\}_{j=1}^{n_i},
\]
where the functions $\xi_i^{[j]}$, called enrichments, are carefully chosen to ``mimic'' the exact solution $u$, locally in $\omega_i$. The approximation subspace of GFEM is as follows:
\begin{equation}
{\cal S}_h= {\cal S}_G = \mbox{span}\{N_i\xi_i^{[j]}:\,i\in I_h,\,\,j=1,\cdots,n_i\}. \label{GFEMspace}
\end{equation}
It is obvious that the GFEM has more DOFs since the additional special functions in $V_i$ are introduced. These special functions and the FEM functions may be almost linearly (or linearly) dependent so that stiffness matrices of GFEM may be badly conditioned or singular even. For instance, in one-dimensional case, if we choose $V_i=\mbox{span}\{1, x-x_i\}$ (linear polynomials), then the stiffness matrix associated with (\ref{GFEMspace}) is singular, see \cite{BaBanOs1} in detail.

An approach to improve the conditioning is to change the FE PU $\{N_i, i\in I_h\}$ to a so-called flat-top PU $\{Q_i^\sigma,\, i\in I_h\}$. We describe the flat-top PU as follows. For every $i\in I_h$, let $\omega_i^\sigma$ be a subset of $\omega_i$, i.e., $\omega_i^\sigma\subset\omega_i$, and there is a constant $\sigma$ independent of $i, h$ such that
\[\mbox{diam}(\omega_i^\sigma) \geq\sigma h,\,\,\forall\,\,i\in I_h,\]
and $\omega_i^\sigma\cap\omega_j^\sigma=\emptyset$ for $i\neq j$. Each PU function $Q_i^\sigma$ is associated with $x_i$, which is of compact support and satisfies
\[Q_i^\sigma(x)=1,\,\,x\in \omega_i^\sigma.\]
namely, the value of $Q_i^\sigma$ in $\omega_i^\sigma$, a neighborhood of $x_i$, is 1. Such a PU is known as the flat-top PU (FT-PU) in the literature. The FT-PU was first used in the PUM in \cite{GriSch}. They were also used in the context of superconvergence study of GFEM in \cite{BaBanOs}. It can be obtained for the FT-PU that
\begin{equation}
\mbox{supp}\{Q_i^\sigma\} \subset \omega_i, \quad \|Q_i^\sigma\|_{L^\infty(\Omega)} \le C_1,\quad \|\nabla Q_i^\sigma\|_{L^\infty(\Omega)} \le C_2 h^{-1}.\label{secondPUcond}
\end{equation}
The approximation space of the FT-PU GFEM is
\begin{equation}
{\cal S}_h= \mbox{span}\{Q_i^\sigma\xi_i^{[j]}:\,i\in I_h,\,\,j=1,\cdots,n_i\}. \label{GFEMspaceFT}
\end{equation}
The FT-PU reduces the linear dependence of shape functions of GFEM. The condition numbers of FT-PU GFEM with polynomial enrichments are proven to be $O(h^{-2})$ in \cite{HLi}, which is of same order as in the FEM. A construction of the FT-PU is given in (\ref{PU}) in Appendix.
\medskip

\noindent\textbf{SGFEM:}

\medskip

Recently, a stable GFEM (SGFEM) was proposed to address the ill-conditioning of GFEM. A GFEM is referred to as SGFEM if (a) it maintains approximation properties of GFEM, and (b) its scaled condition number of stiffness matrix is of same order as that of the FEM. Let ${\cal I}_hf$ be the FE interpolant of a function $f$ defined as
\[{\cal I}_hf(x)=\sum_{i\in I_h}f(x_i)N_i(x),\]
then the approximation space of SGFEM is
\begin{equation}
{\cal S}_h= \mbox{span}\{N_i: \,i\in I_h\} \oplus \mbox{span}\{N_i(\xi_i^{[j]}-{\cal I}_h\xi_i^{[j]}):\,i\in I_h,\,\,j=1,\cdots,n_i\}. \label{SGFEMspace0}
\end{equation}
It was reported \cite{ZhaBaBan} that when applied to enrichments of high order polynomial, the SGFEM (\ref{SGFEMspace0}) has a rank deficiency difficulty in that the element stiffness matrices, associated with the enrichments, are singular. This was overcome in \cite{ZhaBaBan} by changing the FE PU to the FT-PU, and the approximation space of the modified SGFEM is given by
\begin{equation}
{\cal S}_h= \mbox{span}\{N_i: \,i\in I_h\} \oplus \mbox{span}\{Q_i^\sigma(\xi_i^{[j]}-{\cal I}_h\xi_i^{[j]}):\,i\in I_h,\,\,j=1,\cdots,n_i\}. \label{SGFEMspace}
\end{equation}
The SGFEM (\ref{SGFEMspace}) with certain modifications has been applied to crack problems \cite{ZhaBaBan1,GuDuBaBan,SanDu,Zhang0}, interface problems \cite{ZhaBaBan2,BabBanKen}, and high order approximations \cite{ZhaBaBan}, and so on.

\medskip

\noindent\textbf{Fundamental approximation features of GFEM and SGFEM:}

\medskip

The approximation features of GFEM (\ref{GFEMspace}), (\ref{GFEMspaceFT}), and SGFEM (\ref{SGFEMspace0}), (\ref{SGFEMspace}), have been developed early. We present them without proofs, which can be found out in \cite{BaBanOs,BaBanOs1,BaMe,MelBa1,BaBan1,ZhaBaBan}.
\begin{mytheorem}
Let $u$ be the solution to the variational problem (\ref{varprob}) and $u_h$ be the solution to the discretized variational problem (\ref{dvarprob}) produced by the GFEM (\ref{GFEMspace}), (\ref{GFEMspaceFT}), or SGFEM (\ref{SGFEMspace}). Then there is constant $C$ independent of $i$ and $h$ such that for arbitrary $\eta_i\in V_i$, $i\in I_h$,
\begin{equation}\label{FAM}\|u-u_h\|_{{\cal E}(\Omega)}\leq C\Bigg(\sum_{i\in I_i}\Big[\|u-\eta_i\|_{L^2(\omega_i)}^2h^{-2}+|u-\eta_i|_{H^1(\omega_i)}^2\Big]\Bigg)^{1/2}.\end{equation}
\end{mytheorem}

The error estimate (\ref{FAM}) is referred to as the fundamental approximation feature of GFEM and SGFEM, which makes the GFEM and SGFEM yield accurate approximate solutions $u_h$ by designing certain local approximation spaces $V_i$. This is important for many typical non-smooth problems, e.g., crack problems, interface problems.

It is clear that the major structures of GFEM are to enrich the FEM spaces with additional functions. First, either the extra functions and FE functions or the extra functions themselves may be almost linear dependent. Therefore, the conditioning of GFEM could be really bad. This stimulates a great deal of studies to enhance the conditioning of GFEMs, as reviewed in the section of Introduction. Second, the introduction of extra DOFs in the GFEM makes many standard procedures in the FEM complicated, especially in dynamical problems, such as time integration or mass lumping \cite{FrZil,WenTian,MenRet}. Finally, the FT-PU functions used in the GFEM (\ref{GFEMspaceFT}) are piecewise polynomials on each element rather than the polynomials (see (\ref{PU})). Therefore, numerical integrations for the FT-PU functions are more involved than for the FE functions. These difficulties motivate a study on the condensed GFEM in next section.

\section{Condensed GFEM and its approximation features}\label{sec3}

In this section, we propose a condensed GFEM (CGFEM), where (i) the approximation space is a subspace of the GFEM approximation space (\ref{GFEMspace}), and the number of shape functions is the same as that in the preliminary FEM, (ii) the approximation feature is similar with (\ref{FAM}) of the GFEM and SGFEM, and (iii) it is well-conditioned.

On the whole, we will condense the GFEM space (\ref{GFEMspace}) to get a subspace
\begin{equation}\label{condense0}{\cal S}_h = \mbox{span}\{\psi_l: l\in I_h\},\end{equation}
such that any $u\in {\cal E}(\Omega)$ can be approximated by
\begin{equation}\label{mainmode}
{\cal V}_hu(x)=\sum_{l\in I_h}u(x_l)\psi_l(x)
\end{equation}
with a similar approximation feature with (\ref{FAM}), where $\psi_l(x)$ is associated with $x_l$. Remember that $u$ can be approximated by the standard FE interpolant
\begin{equation}\label{FEMmode}
{\cal I}_hu(x)=\sum_{i\in I_h}u(x_i)N_i(x).
\end{equation}
It can be seen that (\ref{mainmode}) and (\ref{FEMmode}) have similar expression forms. However, we will see that (\ref{mainmode}) is an approximation rather than an interpolant. It is clear that the CGFEM (\ref{condense0}) (that is a subspace of (\ref{GFEMspace})) has as many DOFs as those in the FEM since they both have the same index set, $I_h$.

We briefly describe the idea to get (\ref{mainmode}). Remember that $V_i$'s are the local approximation spaces of GFEM (\ref{GFEMspace}). First, the local enrichments in $V_i$ are re-combined according to a least square procedure such that the new enriched functions are characterized by the function values  $u(x_l)$ of some nodes around $x_i$. Then, the new enriched functions are multiplied by the FE PU function $N_i$. Finally, the shape functions $\psi_l$ are derived by collecting the local enriched functions in terms of $u(x_l)$. The detailed constructions are specified below.

For any $x_i$, $i\in I_h$, let $X_i=\{x_l:\,l\in I_i\}$ be a set of nodes in neighborhood of $x_i$ with an index set $I_i$, and assume that $X_i$ is $V_i$-unisolvent, namely,
\begin{equation}\label{Haar}\forall\,\, \eta_i\in V_i, \,\,\eta_i(x_l)=0, x_l\in X_i\,\,\Rightarrow \,\,\,\eta_i\equiv0.\end{equation}
$X_i$ depends on $V_i$, and will be given later. We also assume that there exists a constant $\kappa$ independent of $i$ and $h$ such that
\begin{equation}\label{finiteover}|X_i|\leq \kappa,\,\,\,\,\forall\,\,\,i\in I_h.\end{equation}
For instance, $X_i$ can be the nodes in patch $\omega_i$ (the support of $N_i$).

We next fit $u(x_l), l\in I_i$ using the approximation space $V_i$ by a standard least square (LS) procedure. Specifically, we find a vector ${\bf b} =(b_1, b_2,\cdots,b_{n_i})^T$ such that a function
\begin{equation}\label{LS}
J_i({\bf b}):=\sum_{l\in I_i}\Big({\bf Q}_i^T(x_l){\bf b}-u(x_l)\Big)^2
\end{equation}
is minimal, where ${\bf Q}_i(x)=\Big(\xi_i^{[1]}(x), \xi_i^{[2]}(x),\cdots,\xi_i^{[n_i]}(x)\Big)^T$. Solving the minimal problem (\ref{LS}), we get a local approximation function $u_i$ to $u$ according to the nodes $x_l, l\in I_i$ as follows:
\begin{equation}\label{LS1}u_i(x)=\sum_{l\in I_i}u(x_l){\bf Q}_i^T(x){\bf G}_i^{-1}{\bf Q}_i(x_l),\end{equation}
where \[\,\,\,{\bf G}_i=\sum_{l\in I_i}{\bf Q}_i(x_l){\bf Q}_i^T(x_l)\]
is invertible because $X_i$ is assumed to be $V_i$-unisolvent. Denote
\begin{equation}\label{LSbasis}\tilde{\xi}_i^l(x) := {\bf Q}_i^T(x){\bf G}_i^{-1}{\bf Q}_i(x_l),\,\,l\in I_i\end{equation}
in (\ref{LS1}), we have
\begin{equation}\label{LS2}u_i(x)=\sum_{l\in I_i}u(x_l)\tilde{\xi}_i^l(x).\end{equation}
We note that the LS procedure leads to a reproducing property of $\tilde{\xi}_i^l(x)$:
\begin{equation}\label{reproducing1}\sum_{l\in I_i}\eta_i(x_l)\tilde{\xi}_i^l(x)=\eta_i(x),\,\,\,\forall\,\,\eta_i(x)\in V_i.\end{equation}
To attain a global approximation ${\cal V}_hu$, we employ the PU $\{N_i\}$ to ``paste'' the local approximations $u_i$ together as follows:
\begin{equation}\label{global1}
{\cal V}_hu(x)=\sum_{i\in I_h}N_i(x)u_i(x)=\sum_{i\in I_h}N_i(x)\Big(\sum_{l\in I_i}u(x_l)\tilde{\xi}_i^l(x)\Big).
\end{equation}
We rewrite (\ref{global1}) by combining the multipliers of each $u(x_l)$ as
\begin{equation}\label{global2}
{\cal V}_hu(x)=\sum_{i\in I_h}\sum_{l\in I_i}u(x_l)N_i(x)\tilde{\xi}_i^l(x)=\sum_{l\in I_h}u(x_l)\psi_l(x),
\end{equation}
where
\[\psi_l(x):=\sum_{i\in L_l}N_i(x)\tilde{\xi}_i^l(x),\,\,\,L_l=\{i: l\in I_i\}.\]
Now, we obtain an approximation ${\cal V}_hu$ (\ref{mainmode}) of $u$. The shape functions $\psi_l$ are obtained by recombining $N_i(x)\tilde{\xi}_i^l$ that are constructed from the basis of GFEM (\ref{GFEMspace}), see (\ref{LSbasis}) for the definition of $\tilde{\xi}_i^l$. Therefore, the approximation space
\begin{equation}\label{CGFEMspace}
{\cal S}_h={\cal S}_{CG}:=\{\psi_l:\,l\in I_h\}
\end{equation}
is a subspace of the preliminary GFEM space (\ref{GFEMspace}), namely,
\begin{equation}\label{contain}
{\cal S}_{CG}\subset {\cal S}_G.
\end{equation}
We refer to the GFEM with the approximation space (\ref{CGFEMspace}) as condensed GFEM (CGFEM). Clearly, the CGFEM has as many DOFs as FE nodes. From point of view of algorithm, the CGFEM needs the local approximation space $V_i$ and $V_i$-unisolvent set $X_i$. $V_i$ is the same as in the GFEM (\ref{GFEMspace}), while $X_i$ is designed according to $V_i$ and problem-dependent, which will be presented in the next section. We will next prove that the CGFEM (\ref{CGFEMspace}) maintains a similar approximation feature to (\ref{FAM}) of the preliminary GFEM (\ref{CGFEMspace}).

\begin{mytheorem}\label{maintheo}
Let $u$ be the solution to the variational problem (\ref{varprob}) and $u_h$ be the solution to the discretized variational problem (\ref{dvarprob}) produced by the CGFEM (\ref{CGFEMspace}). Assume that for any $i\in I_h$, $X_i$ is $V_i$-unisolvent (\ref{Haar}). Then there is constant independent of $h$ such that for arbitrary $\eta_i\in V_i$, $i\in I_h$,
\begin{eqnarray}\label{FAM1}\|u-u_h\|_{{\cal E}(\Omega)}&\leq& C\Bigg(\sum_{i\in I_h} \Big[\|u-\eta_i\|_{L^2(\omega_i)}^2h^{-2}+|u-\eta_i|_{H^1(\omega_i)}^2\Big] \nonumber\\
&&+ \sum_{i\in I_h} \Big[ \sum_{l\in I_i}|u(x_l)-\eta_i(x_l)|^2\Big(\|\tilde{\xi}_i^l\|_{L^2(\omega_i)}^2h^{-2}+|\nabla\tilde{\xi}_i^l|_{L^2(\omega_i)}^2\Big)\Big] \Bigg)^{1/2}.\end{eqnarray}
\end{mytheorem}

\begin{proof}
Since $X_i$ is $V_i$-unisolvent for all $i\in I_h$, the shape functions $\psi_l$ in (\ref{global2}) can be constructed uniquely through the local LS procedure (\ref{LS}). Let ${\cal V}_hu:=\sum_{l\in I_h}u(x_l)\psi_l(x)$ as in (\ref{global2}),
using \eqref{global1}, \eqref{global2}, PU property of $N_i$, (\ref{Shacond}), and \eqref{LS2}, we have
\begin{eqnarray}
\|u-{\cal V}_hu\|_{{\cal E}(\Omega)}^2&=&\|\sum_{i\in I_h}N_i(u-u_i)\|_{{\cal E}(\Omega)}^2=\|\sum_{i\in I_h}N_i\Big(u-\sum_{l\in I_i}u(x_l)\tilde{\xi}_i^l\Big)\|_{{\cal E}(\Omega)}^2\nonumber
\\&\leq& 2\sum_{i\in I_h} \Bigg[ \|u-\sum_{l\in I_i}u(x_l)\tilde{\xi}_i^l\|_{L^2(\omega_i)}^2|N_i|_{H^1(\omega_i)}^2
+|u-\sum_{l\in I_i}u(x_l)\tilde{\xi}_i^l|_{H^1(\omega_i)}^2 \|N_i\|_{L^2(\omega_i)}^2 \Bigg]\nonumber\\
&\leq& C\sum_{i\in I_h} \Bigg[ \|u-\sum_{l\in I_i}u(x_l)\tilde{\xi}_i^l\|_{L^2(\omega_i)}^2h^{-2}
+|u-\sum_{l\in I_i}u(x_l)\tilde{\xi}_i^l|_{H^1(\omega_i)}^2 \Bigg].\label{mainest1}
\end{eqnarray}
According to the reproducing property (\ref{reproducing1}), we have for any $\eta_i\in V_i$,
\[u-\sum_{l\in I_i}u(x_l)\tilde{\xi}_i^l = u-\eta_i - \sum_{l\in I_i}(u(x_l)-\eta_i(x_l))\tilde{\xi}_i^l.\]
Using this in (\ref{mainest1}) yields
\begin{eqnarray}
\|u-{\cal V}_hu\|_{{\cal E}(\Omega)}^2
&\leq& C\sum_{i\in I_h} \Big[\|u-\eta_i\|_{L^2(\omega_i)}^2h^{-2}+|u-\eta_i|_{H^1(\omega_i)}^2\Big] \nonumber\\
&&+ C\sum_{i\in I_h}\Big[  \sum_{l\in I_i} |u(x_l)-\eta_i(x_l)|^2\Big(\|\tilde{\xi}_i^l\|_{L^2(\omega_i)}^2h^{-2}+\|\nabla\tilde{\xi}_i^l\|_{L^2(\omega_i)}^2\Big)\Big] .\label{mainest2}
\end{eqnarray}
Now, the desired result (\ref{FAM1}) is obtained from a fact $\|u-u_h\|_{{\cal E}(\Omega)}\leq \|u-{\cal V}_hu\|_{{\cal E}(\Omega)}$.
\end{proof}

The error estimate (\ref{FAM1}) is the approximation property of CGFEM, which is similar to (\ref{FAM}) of GFEM and SGFEM. A difference consists in that the regularity of local functions $\|\tilde{\xi}_i^l\|_{L^2(\omega_i)}$ and $\|\nabla\tilde{\xi}_i^l\|_{L^2(\omega_i)}$ is involved in (\ref{FAM1}). This implies that estimates on the regularity of $\tilde{\xi}_i^l$ are important for the CGFEM. We assume the regularity as follows:
\begin{equation}
\quad \|\tilde{\xi}_i^l\|_{L^\infty(\Omega)} \le C_1,\quad \|\nabla \tilde{\xi}_i^l\|_{L^\infty(\Omega)} \le C_2 h^{-1},\,\,\,\forall\,\,\,i\in I_h,\,\,l\in I_i, \label{regu}
\end{equation}
where $C_1,\, C_2$ are generic constants independent of $i$ and $h$. We see that the approximation feature (\ref{FAM1}) has the similar behaviour with (\ref{FAM}) if the regularity condition (\ref{regu}) is satisfied. In next section, we will prove the regularity result (\ref{regu}) for the CGFEM when the high order polynomials serve as the enrichments, namely, $V_i={\cal P}_k$. Based on that, the optimal convergence order $O(h^k)$ under energy norm will also be proven.

\begin{remark}
The similar ideas to develop the CGFEM can be found out in intrinsic GFEM (PUM) \cite{FrBelInt,FrBelIntPU,Fries}, PUM-RBF methods \cite{MolFer,Wen1}, extra-DOF-free GFEM \cite{Tian}, and so on. The intrinsic GFEM is constructed using the moving least square procedures, which causes the numerical integration difficulty as in meshless methods \cite{BBOZ}. The extra-DOF-free GFEM employs a selectively interpolating least square approach to set up the local approximation functions, which poses a strong constraint for linear functions and thus leads to accuracy loss when the linear functions serve as the enrichments. Moreover, to our best knowledge, a rigorous theoretical analysis about the approximation features (\ref{FAM1}) has not made in those studies.   \myend
\end{remark}

\begin{remark}
As with other methods based on the least-square schemes, the CGFEM does not satisfy the Kronecker condition, i.e., $\psi_l(x_i)\neq \delta_{li}$. Therefore, imposing essential boundary conditions (EBC) is not straightforward as in the FEM. This occurs to meshless methods typically \cite{BaBanOs,NguRab}. Therefore, the well-developed techniques to impose the EBC in meshless methods, such as Nitsche's method, coupling with the FEM, could also be applied to the CGFEM efficiently. We refer to \cite{HueFer,FerHue,BaBanOs,BaBanOs1,GriSch1,Zhang} for detailed discussions on the EBC. We will address the EBC for CGFEM specifically in a forthcoming study.  \myend
\end{remark}

In the end of this section, we present a one-dimensional (1D) example to illustrate the CGFEM shape functions intuitively. Let $\{x_i\}$ be the 1D FE nodes with $x_{i-1}<x_i$, $N_i$ be the standard FE hat function with respect to $x_i$, and the enriched spaces $V_i = {\cal P}_k$. To satisfy the ${\cal P}_k$-unisolvent condition (\ref{Haar}), $k=1, 2, 3$, let
\[X_i=\left\{\begin{array}{ll}\{x_{i-1}, x_i, x_{i+1}\}, & k=1, 2\\
\{x_{k-2}, x_{i-1}, x_i, x_{i+1}, x_{k+2}\}, & k=3.\end{array}\right. \]
The associated CGFEM shape functions for $k=1, 2, 3$ are displayed in Fig. \ref{shafun}. The standard FE functions of degree $k$ are also drawn in Fig. \ref{shafun} for comparison.

\begin{figure}
\centering
\includegraphics[width=.90\linewidth]{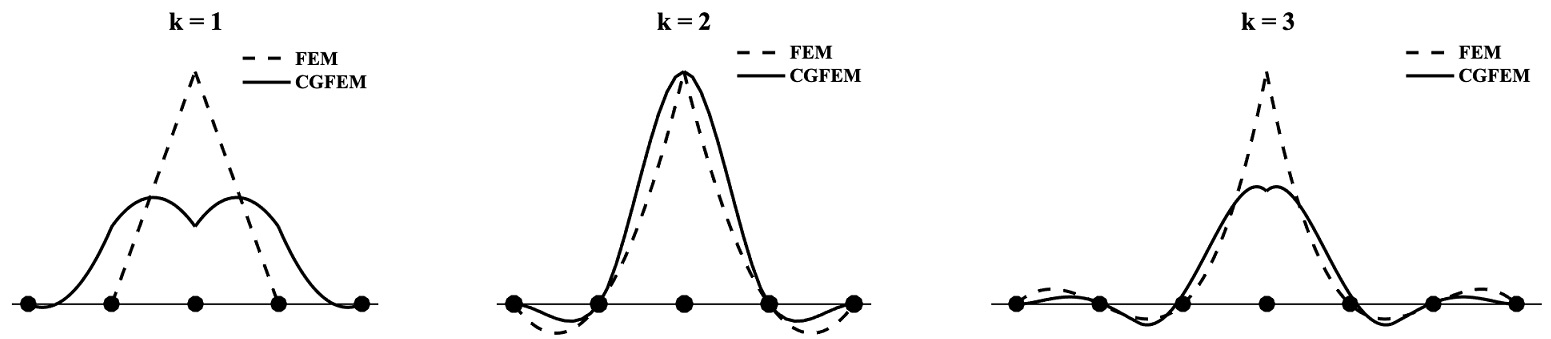}
\caption{An illustration of shape functions in a one-dimensional situation. Solid line: the CGFEM shape functions with the polynomial enrichments ${\cal P}_k$; dash line: FE shape functions of degree $k$. Left: $k=1$, middle: $k=2$, right: $k=3$.
}\label{shafun}
\end{figure}

\section{Applications}\label{sec4}

The CGFEM can be applied to various smooth and non-smooth problems, such as, high order approximations \cite{ZhaBaBan}, crack problems \cite{BelGra,FleChu,LabJul,GuDuBaBan,WenTian,ZhaBaBan1}, interface problems \cite{BabBanKen,Fries,ZhaBaBan2}, problems with multiple voids and inclusions \cite{NguRab}. In this paper, we focus on a high order approximation problem where the high order polynomials serve as enrichments and a Poisson crack problem. For the former, we will prove the regularity (\ref{regu}) of shape functions and the optimal convergence orders $O(h^k)$ under an energy norm. For the crack problem, we will construct an enrichment scheme for the CGFEM and verify its optimal convergence order $O(h)$ numerically in the next section. In both problems, the conditioning of CGFEM will be shown to be of the same order as that of the FEM.

\subsection{Enrichments of high order polynomials}

\begin{figure}
\centering
\includegraphics[width=.90\linewidth]{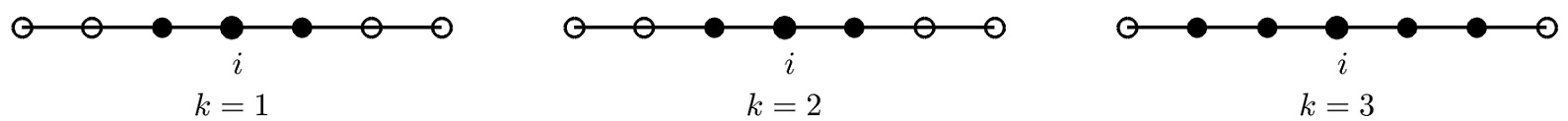}\\\vspace{0.2cm}\includegraphics[width=.90\linewidth]{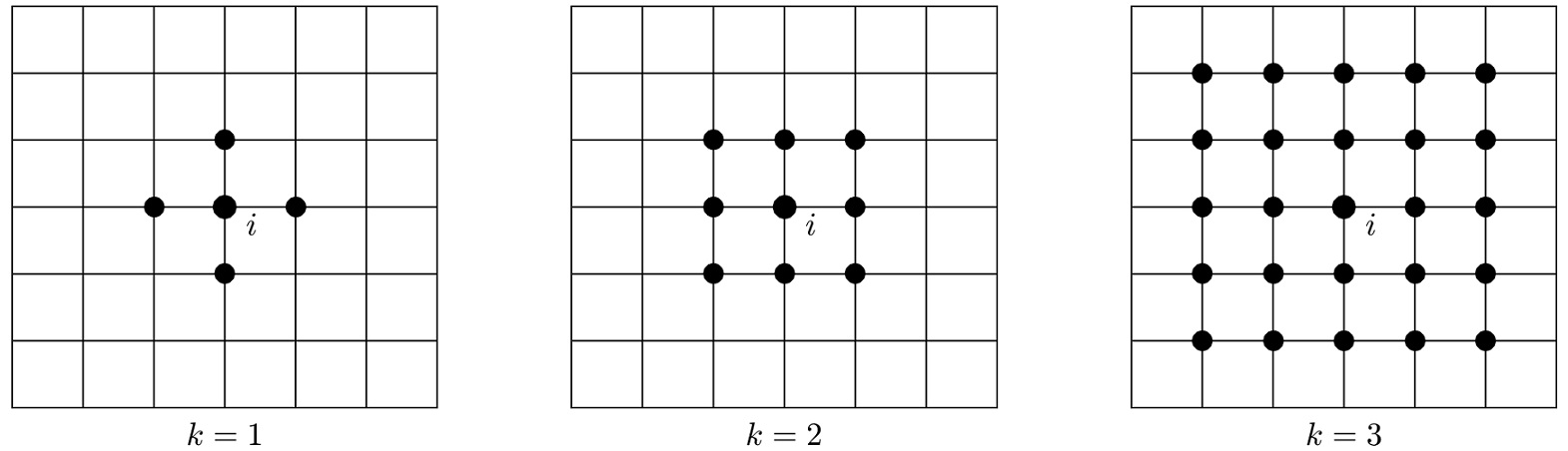}
\caption{An illustration of $X_i$ associated with a node $x_i$, where the big dots are the nodes in $X_i$. Three above: one-dimensional situation; three below: two-dimensional situation; left two: $k=1$; middle two: $k=2$; right two: $k=3$. $V_i={\cal P}_k$, and $X_i$ in the left two, middle two, and right two are $V_i$-unisolvent for $k=1, 2, 3$, respectively.
}\label{shafun}
\end{figure}

Recall that $N_i$ is the standard FE hat function (linear or bilinear) with respect to $x_i$, $\omega_i$ is the patch, and $e_s$'s are the elements. For any $i\in I_h$, let the local enriched space
\[V_i ={\cal P}_k=\mbox{span}\{P_i^\alpha = (\frac{x-x_i}{h})^\alpha\, :\, |\alpha|=0,1,\dots,k\},\]
where $\alpha$ is a multi-index. Denote the dimension of $V_i$ by $\chi$, then clearly,
\begin{equation}
\chi ={{k+d} \choose k}. \label{defk}
\end{equation}
The ${\cal P}_k$-unisolvent set $X_i$ is defined as follows,
\begin{equation}\label{HOPUS}
X_i = \left\{\begin{array}{ll}X_i^1 = \{x_i\}\cup\{x_j: x_j \,\,\mbox{is connected to} \,\,x_i \,\,\mbox{through a side}\}, & k=1,\\
 X_i^2 = \{x_j: x_j\in \omega_i\}, & k=2,\\
 X_i^k = \{x_j: x_j \in e_s \,\,\,\mbox{and}\,\,\, e_s\cap X_i^{k-1}\neq \emptyset\}, & k\geq 3.\end{array}\right.
\end{equation}
$X_i, k=1, 2, 3$ in 1D and 2D are exhibited in Fig. \ref{shafun}.

We first establish the regularity (\ref{regu}) of local basis functions $\tilde{\xi}_i^l\in {\cal P}_k, x_l\in X_i$; $X_i$ is defined in (\ref{HOPUS}). A cone condition of the domain $\Omega$ is presented, which will be used in the proof. A domain $\Omega$ is said to satisfy a cone condition with angle $\vartheta$ and radius $\varrho$ if for each $x \in \Omega$, there exists a ${\bf c}\in {\mathbb R}^d$ with $\|{\bf c}\|_2=1$ such that a cone $C(x,{\bf c},\vartheta,\varrho)\subset\Omega$,
where
\[C(x,{\bf c},\vartheta,\varrho):=B_\varrho(x)\{y\in{\mathbb R}^d: (y-x)\bullet{\bf c}>\|y-x\|_2\cos \theta\}\]
and $B_\varrho(x)$ is a ball of radius $\varrho$ centered at $x$. The definition of cone condition is referred to \cite{Mel}.

\begin{myprop}
Assume that $\Omega$ satisfies the cone condition with angle $\theta$ and radius $r$, and $X_i$ is ${\cal P}_k$-unisolvent. Then the regularity (\ref{regu}) of local LS functions in (\ref{LSbasis}) is satisfied for the CGFEM with the enrichments ${\cal P}_k$, namely,
for any $i\in I_h$ and $l\in I_i$, there is constant $C$ independent of $i$ and $h$ such that
\begin{equation}
\quad \|\tilde{\xi}_i^l\|_{L^\infty(\omega_i)} \le C_1,\quad \|\nabla \tilde{\xi}_i^l\|_{L^\infty(\omega_i)} \le C_2 h^{-1}.\label{reguHO}
\end{equation}
\end{myprop}

\begin{proof}
Remembering the definition of $\tilde{\xi}_i^l$ in (\ref{LSbasis}) and (\ref{LS2}), we know that in this situation,
\[\tilde{\xi}_i^l(x) := {\bf Q}_i^T(x){\bf G}_i^{-1}{\bf Q}_i(x_l),\]
where
\[{\bf Q}_i(x)=\Big(P_i^\alpha = (\frac{x-x_i}{h})^\alpha: |\alpha|=0,1,\cdots, k\Big)^T\,\mbox{ and }\,{\bf G}_i=\sum_{l\in I_i}{\bf Q}_i(x_l){\bf Q}_i^T(x_l)\]
It is noted that ${\bf Q}_i(x)$ is a column vector of length $\chi$ (\ref{defk}). First, ${\bf G}_i$ is a positive definite matrix. Indeed, let ${\bf a}=(a_\alpha: |\alpha|=0,1,\cdots, k)^T$ be an arbitrary vector in ${\mathbb R}^\chi$, and
\begin{equation}\label{posi}
{\bf a}^T{\bf G}_i{\bf a}=\sum_{l\in I_i}|\pi_i(x_l)|^2\, \mbox{ and }\, \pi_i(x):=\sum_{|\alpha|=0}^k a_\alpha P_i^\alpha(x).
\end{equation}
Then ${\bf a}^T{\bf G}_i{\bf a}=0$ implies that $\pi_i(x_l)=0$ for $x_l, l\in I_i$. Since $X_i=\{x_l, l\in I_i\}$ is ${\cal P}_k$-unisolvent, and $P_i^\alpha$ are independent, we have $\pi_i(x)\equiv 0$, and thus ${\bf a}=0$. This means that ${\bf G}_i$ is a positive definite matrix.

Next if we prove the minimal eigenvalue $\lambda_{\min}$ of ${\bf G}_i$ has a lower bound that is independent of $h$ and $i$, then we can get (\ref{reguHO}). In fact, for any $x_l\in X_i$, the vector ${\bf G}_i^{-1}{\bf Q}_i(x_l)$ satisfies
\[\|{\bf G}_i^{-1}{\bf Q}_i(x_l)\|_2 \leq \lambda_{\min}^{-1} \|{\bf Q}_i(x_l)\|_2\leq \lambda_{\min}^{-1}C,\]
based on the definition (\ref{HOPUS}) of $X_i$, where $C$ is constant independent of $h$ and $i$. From this and the formulation of $\tilde{\xi}_i^l$, we get (\ref{reguHO}).

Below, we will prove that $\lambda_{\min}$ has a lower bound that is independent of $h$ and $i$. For any vector ${\bf a}=(a_\alpha: |\alpha|=0,1,\cdots, k)^T$, we have from (\ref{posi}) that
\[{\bf a}^T{\bf G}_i{\bf a}\geq \max_{l\in I_i}|\pi_i(x_l)|^2.\]
According to Lemma 2.24 in \cite{Mel} (page 19), there is a constant $C_1$ and $C_2$ independent of $h$ and $i$ such that
\[\|\pi_i\|_{L^\infty(B(x_i,C_1h))}\leq C_2 \max_{l\in I_i}|\pi_i(x_l)|.\]
Then, we get
\begin{equation}\label{posi1}{\bf a}^T{\bf G}_i{\bf a}\geq \frac{1}{C_2^2}\|\pi_i\|_{L^\infty(B(x_i,C_1h))}^2.\end{equation}
Making a scaling argument by $\xi=\frac{x-x_i}{h}$, we have $\tilde{\pi}_i(\xi) = \pi_i(x(\xi))$ is a polynomial in $B(O,C_1)$ for any $i\in I_h$. Using a norm equivalence of polynomial space in $B(O,C_1)$, we get
\[\|\tilde{\pi}_i(\xi)\|_{L^\infty(B(O,C_1))}\geq C \|{\bf a}\|_2.\]
By this and (\ref{posi1}), we get
\[{\bf a}^T{\bf G}_i{\bf a}\geq \frac{1}{C_2^2}\|\pi_i\|_{L^\infty(B(x_i,C_1h))}^2 = \frac{1}{C_2^2}\|\tilde{\pi}_i\|_{L^\infty(B(O,C_1))}^2\geq \frac{C}{C_2^2} \|{\bf a}\|_2,\]
which implies that $\lambda_{\min}$ has a lower bound $\frac{C}{C_2^2}$ that is independent of $h$ and $i$.
\end{proof}

\begin{mytheorem}
Let $u\in H^{k+1}(\Omega)$ be the solution to the variational problem (\ref{varprob}) and $u_h$ be the solution to the discretized variational problem (\ref{dvarprob}) produced by the CGFEM (\ref{CGFEMspace}) with the local space $V_i={\cal P}_k$. Assume that $\Omega$ satisfies the cone condition with angle $\theta$ and radius $r$, and $X_i$ is ${\cal P}_k$-unisolvent. Then there is constant independent of $h$ such that
\begin{eqnarray}\label{errorder}\|u-u_h\|_{{\cal E}(\Omega)}&\leq& Ch^{k}\|u\|_{H^{k+1}(\Omega)}.\end{eqnarray}
\end{mytheorem}

\begin{proof}
Let $\eta_i\in V_i$ be the Taylor polynomial $T_i^ku(x)$ of degree $k$ of $u$ at point $x_i$ in (\ref{FAM1}). Using a standard error estimate (\cite{Cia})
\[\|u-\eta_i\|_{L^2(\omega_i)}\leq Ch^{k+1}\|u\|_{H^{k+1}(\omega_i)},\,|u-\eta_i|_{H^1(\omega_i)}\leq Ch^k \|u\|_{H^{k+1}(\omega_i)},\,\|u-\eta_i\|_{L^\infty(\omega_i)}\leq Ch^{k+\frac{2-d}{2}}\|u\|_{H^{k+1}(\omega_i)}\]
and the regularity result (\ref{reguHO}) in (\ref{FAM1}), we get the desired result (\ref{errorder}).

\end{proof}

\subsection{A Poisson crack problem}
Let $\Omega$ be a bounded cracked domain in ${\mathbb R}^2$ with the boundary $\partial\overline{\Omega}$, and $\vec{n}$ be
the unit outward normal vector to $\partial\overline{\Omega}$. A straight crack in $\Omega$ is denoted by $\Gamma_O$ with a crack tip $O$, as shown in Fig. \ref{CPdomain} left. We consider a Poisson model problem:
\begin{equation}\label{craequ}
-\triangle u = f \,\,\mbox{in}\,\,\,\Omega,\,\,\,\,\frac{\partial u}{\partial \vec{n}} =g\,\,\mbox{on}\,\,\partial\overline{\Omega},
\end{equation}
with a traction free condition
\[\frac{\partial u}{\partial \vec{n}_O} = 0 \mbox{ on } \Gamma_O,\]
where $\vec{n}_O$ is unit vector normal to $\Gamma_O$. The solution $u$ to (\ref{craequ}) can be decomposed \cite{ZhaBaBan1} into
\begin{equation}\label{exasol}
u = \sum_{i=1}^\infty \chi_i r^{\frac{2i-1}{2}}\sin(\frac{2i-1}{2}\theta) + u_0
\end{equation}
where $(r,\theta)$ is the polar coordinate with the crack tip $O$ serving as the pole and the opposite direction of crack line as  polar line, and $\chi_i$ are the constants, and $u_0$ is a smooth on the closure $\overline{\Omega}$ of cracked domain $\Omega$ (noting that $\overline{\Omega}$ is not cracked).

\begin{figure}
\begin{center}
\includegraphics[scale=.30]{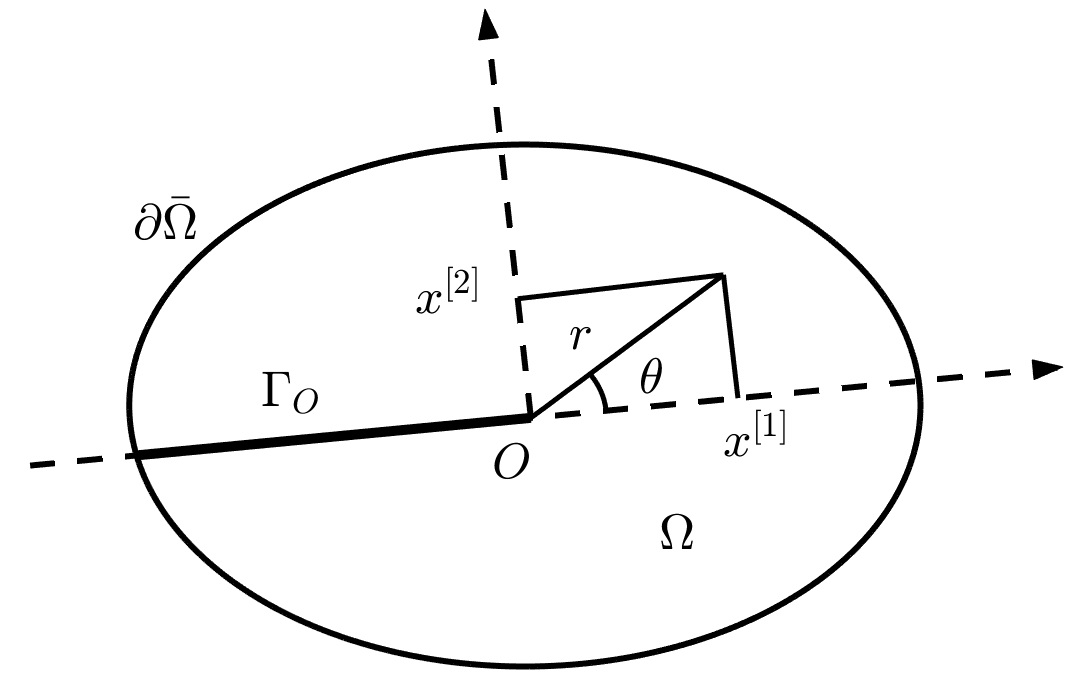} \includegraphics[scale=.30]{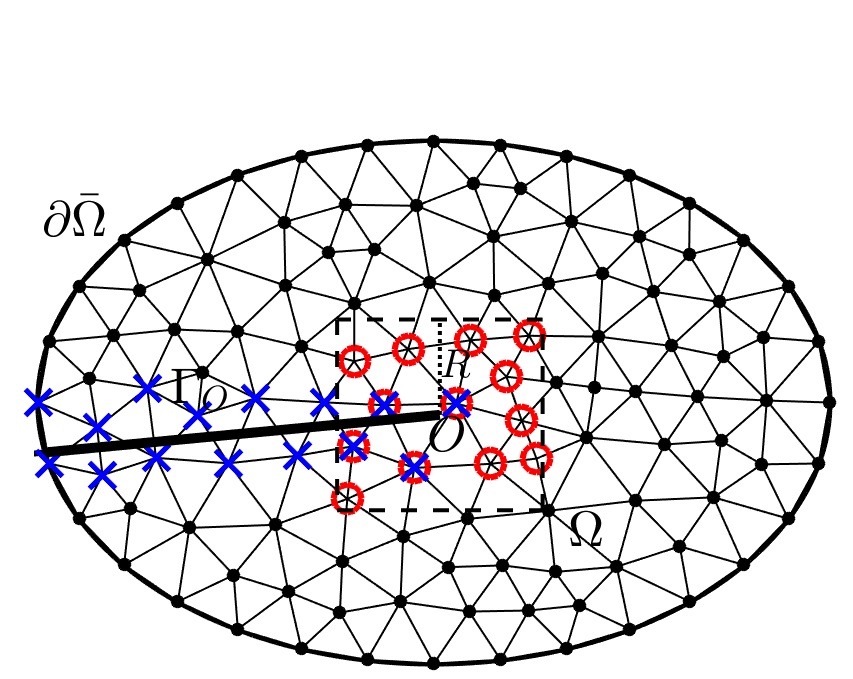}
\caption{{\footnotesize Left: the domain $\Omega$ with a crack line $\Gamma_O$ and a tip $O$, where $x=(x^{[1]}, x^{[2]})$ and $(r,\theta)$ are the Cartesian and polar coordinate systems with the origin $O$, respectively.
Right: the enrichment scheme of CGFEM: the nodes in $I_h^1$ and $I_h^2$ are marked by ``$\times$'' and ``$\circ$'', respectively, the nodes in $I_h^2\setminus I_h^1$ are marked by ``$\circ$'' without ``$\times$'',
and the nodes in $I_h\setminus(I_h^1\cup I_h^2)$ are marked by ``$\bullet$'' (not ``$\circ$'' or ``$\times$'').}}  \label{CPdomain}
\end{center}
\end{figure}

Still, $N_i$ is the standard FE hat function (linear or bilinear) with respect to $x_i$, $\omega_i$ is the patch, and $e_s$'s are the elements. To describe the enrichment scheme for the crack problem (\ref{craequ}), we define
\begin{equation}
I_h^1 = \{i\in I_h:\,x_i\in e_s \,\,\mbox{and}\,\,e_s\cap \Gamma_O\neq\emptyset\}\,\,\,\mbox{and}\,\,\,I_h^2 = \{i\in I_h:\,x_i\in D(O,R)\} \label{Ih12}
\end{equation}
where $D(O,R)$ is a square centered at $O$ with a side length $2R$; $R$ is a constant independent of $h$ and $i$. See Fig. \ref{CPdomain} right for an exhibition of $I_h^1$ and $I_h^2$. We only consider approximation of degree 1 in this problem. Denote $S_{\frac{1}{2}}=\sqrt{r}\sin\frac{\theta}{2}$, and the local enrichment space is defined as follows,
\begin{equation}\label{CPenr}
V_i = \left\{\begin{array}{ll}\mbox{span}\{1, \frac{x^{[1]}-x_i^{[1]}}{h}, \frac{x^{[2]}-x_i^{[2]}}{h}\}, & i\in I_h\setminus(I_h^1\cup I_h^2),\\
 \mbox{span}\{1, \frac{x^{[1]}-x_i^{[1]}}{h}, \frac{x^{[2]}-x_i^{[2]}}{h}, S_{\frac{1}{2}} \}, & i\in I_h^2\setminus I_h^1,\\
 \mbox{span}\{1, \frac{x^{[1]}-x_i^{[1]}}{h}, \frac{x^{[2]}-x_i^{[2]}}{h}, S_{\frac{1}{2}}, S_{\frac{1}{2}}\frac{\langle x-x_i, \vec{n}_O^\perp\rangle}{h} \}, & i\in I_h^1,\end{array}\right.
\end{equation}
where $x:=(x^{[1]}, x^{[2]})$ is the Cartesian coordinates in ${\mathbb R}^2$, and $\langle x,y\rangle :=x^{[1]}y^{[1]}+ x^{[2]}y^{[2]}$. The theoretical analysis on approximation properties of $S_{\frac{1}{2}}$ and $S_{\frac{1}{2}}\frac{\langle x-x_i, \vec{n}_O^\perp\rangle}{h}$ is referred to \cite{ZhaBaBan1}. $V_i$-unisolvent set $X_i$ for (\ref{CPenr}) is defined by
\begin{equation}\label{CPUS}X_i = \{x_j:, x_j\in\omega_i\}.\end{equation}

In the next section, we will numerically verify that the CGFEM with $V_i$ and $X_i$ for the crack problem (\ref{craequ}) yields the optimal convergence order $O(h)$ under energy norm, and moreover, its scaled condition numbers are of same order as those of the standard FEM.

\begin{remark}
If the regularity (\ref{regu}) of the local enriched functions is proven, we can prove the optimal convergence $O(h)$ under energy norm using arguments in \cite{ZhaBaBan1} (Theorem 5.2). The proof of (\ref{regu}) depends on not only the singular factor $S_{\frac{1}{2}}$ but also relative positions between the patch $\omega_i$ and crack line $\Gamma_O$. It will be very technical, and we will not carry out it in this paper. The optimal convergence order $O(h)$ will be justified numerically in the next section. To highlight the main idea, we consider a Poisson crack problems, and an extension to the elasticity crack problems will be investigated similarly in a forthcoming study.  \myend
\end{remark}

\section{Numerical Experiments}                                                                                                                                                                                                                                  The numerical experiments are executed to test effectiveness of CGFEM and verify the theoretical results. The comparisons of CGFEM with the existing GFEM and SGFEM are made by computing their relative energy error
\[
EE:=\frac{\|u-u_h\|_{{\cal E}(\Omega)}}{\|u\|_{{\cal E}(\Omega)}}
\]
and scaled condition numbers (SCN) of stiffness matrices, where $u_h$ is an approximate solution derived from the above-mentioned FEM, GFEM, SGFEM, or CGFEM. The SCN of a matrix $A$ is defined as the ratio of largest and smallest eigenvalues of ${\bf D}^{-\frac{1}{2}}{\bf A}{\bf D}^{-\frac{1}{2}}$,
where ${\bf D}$ is a diagonal matrix composed of diagonal entries of ${\bf A}$.

\subsection{High order polynomial approximations for 2D smooth problems}
We first test the high order approximations for 2D smooth problems. The following methods will be compared:
\begin{itemize}
  \item FEM: the standard bilinear FEM of degree $k$, where the shape functions are constructed from an reference element and isoparametric mappings between the reference and real elements, see Fig. \ref{meshes}.
  \item f.t.GFEM: the conventional GFEM (\ref{GFEMspaceFT}) with the flat-top PU (\ref{PU}) and $V_i={\cal P}_k$,
  \item SGFEM: the SGFEM (\ref{SGFEMspace}) with the flat-top PU (\ref{PU}) and $V_i={\cal P}_k$,
  \item CGFEM: the CGFEM (\ref{CGFEMspace}) with the piecewise linear (or bilinear) PU $N_i$ and $V_i={\cal P}_k$.
\end{itemize}
We note that the tested FEM (the first item above) is the high order FEM, and its shape functions are different from the FE functions $N_i$ in (\ref{Shacond}), the standard FE hat functions (linear or bilinear) serving as the PU functions for the CGFEM (\ref{CGFEMspace}). It was shown in \cite{HLi,ZhaBaBan} that the SCNs of f.t.GFEM (\ref{GFEMspaceFT}) and SGFEM (\ref{SGFEMspace}) are of same order as those of the FEM. We will see below that in addition to the optimal approximations, the SCNs of proposed CGFEM (\ref{CGFEMspace}) are also of same order as those of the FEM.

\medskip

Let $\Omega = (0,1)\times(0,1)$, $\Gamma=\partial\Omega$, and the model problem we consider is
\begin{equation}\label{EStrongModelin2d}
- \triangle u = f \,\,\, \mbox{in} \,\,\Omega\,\,\,\mbox{and}\,\,\,\frac{\partial u}{\partial n}(x)=g(x)\,\,\,\mbox{on}\,\, \Gamma.
\end{equation}
Assume $u(x)=e^{2x^{[1]}+x^{[2]}}$ is an exact solution, where $x=(x^{[1]},x^{[2]})$ is the Cartesian coordinate in ${\mathbb R}^2$, and $f(x)$ and $g(x)$ are calculated through equation (\ref{EStrongModelin2d}) using $u$.

Let $h=\frac{1}{N}$ be the mesh-size parameter; $N$ is a positive integer. We consider two kinds of meshes on $[0, 1]\times [0,1]$ as follows: \begin{itemize}
\item uniform mesh (UM): the elements $e_{ij}=[\frac{i}{N}, \frac{i+1}{N}]\times[\frac{j}{N}, \frac{j+1}{N}], i, j=0,1,2,\cdots,N-1$, and the nodes $x_{ij}=(\frac{i}{N}, \frac{j}{N})$, $i, j=0,1,2,\cdots,N$, see Fig. \ref{meshes} left; \item perturbed mesh (PM): the uniform mesh above is perturbed, and the nodes $x_{ij} = (\frac{i}{N}, \frac{j}{N}) + 0.1h\epsilon_{ij}, \,\,i, j=1,\cdots,N-1$, ${x}_{0j}=(0, \frac{j}{N})$, ${x}_{Nj}=(1, \frac{j}{N})$, ${x}_{i0}=(\frac{i}{N}, 0)$, ${x}_{iN}=(\frac{i}{N}, 1)$, where $\epsilon_{ij}$ is a random number produced from a uniform distribution on $[-0.5,0.5]\times[0,5,0.5]$, see Fig. \ref{meshes} right.
\end{itemize}
Let $N_{ij}$ and $Q_{ij}^\sigma, i,j=0,1,\cdots,N$ be the FE hat functions (bilinear) and the flat-top PU functions with a parameter $\sigma$, respectively, associated with $x_{ij}$. A construction of $Q_{ij}^\sigma$ is given
in (\ref{PU}) in Appendix, and we use $l = 1$ and $\sigma = 0.2$ in the PU functions (\ref{PU}) for the tests below. For any $i, j$, the enrichment space for the GFEM, SGFEM, and CGFEEM is set by
\[V_{ij}={\cal P}_k=\mbox{span}\{\big(\frac{x-x_{ij}}{h}\big)^\alpha:\,|\alpha|=0,1,\cdots,k\},\]
and $V_{ij}$-unisolvent set $X_{ij}$, used to construct the CGFEM, is defined in (\ref{HOPUS}) and is exhibited in Fig. \ref{shafun}.

\begin{figure}
\centering
\includegraphics[width=.40\linewidth]{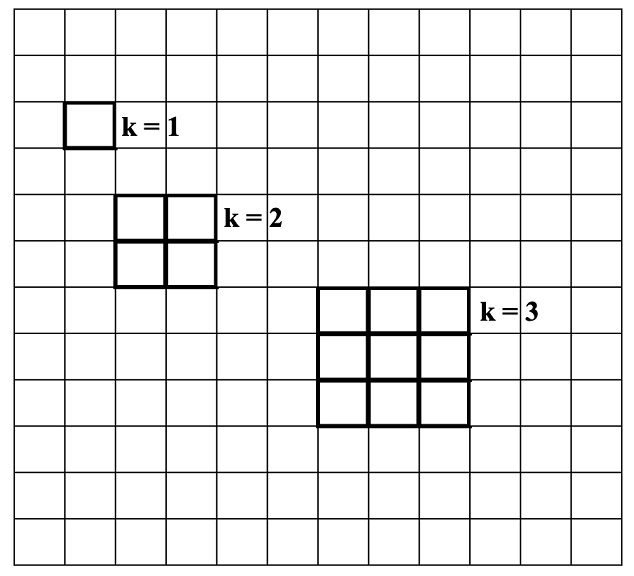}\hspace{0.5cm}\includegraphics[width=.40\linewidth]{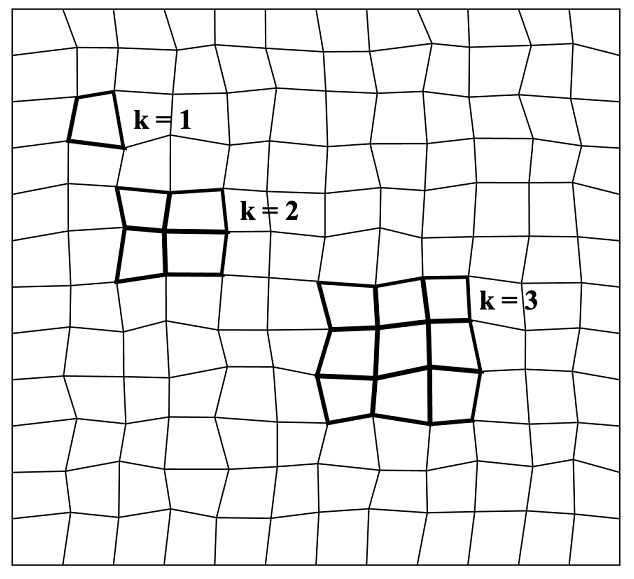}
\caption{An illustration of two-dimensional meshes with $N=12$. Left: uniform mesh, right: perturbed mesh.
}\label{meshes}
\end{figure}

\medskip

\begin{figure}
\centering
\includegraphics[width=.33\linewidth]{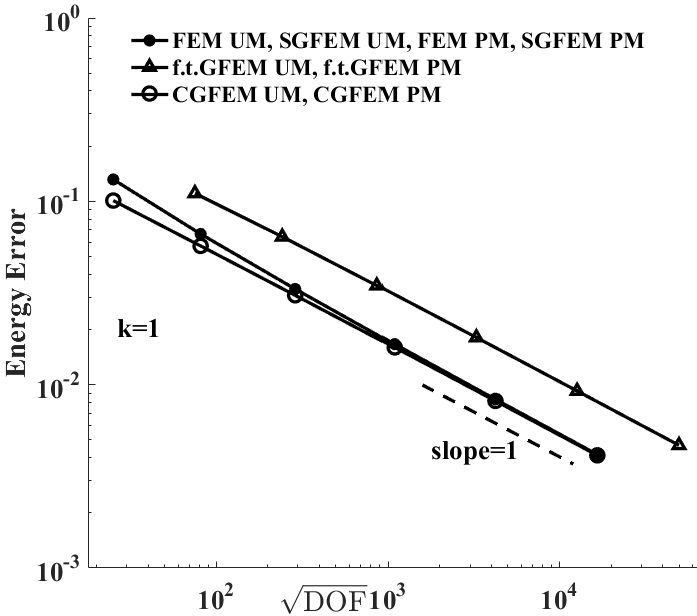}\includegraphics[width=.33\linewidth]{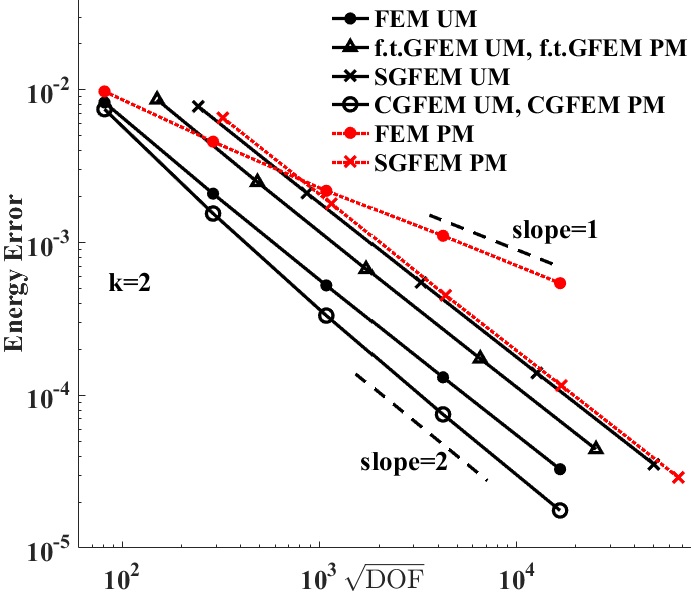}\includegraphics[width=.33\linewidth]{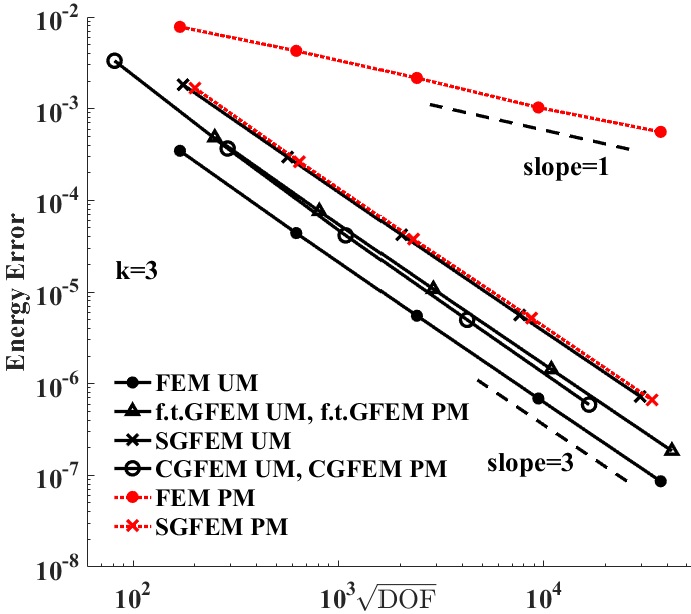}
\caption{ The relative errors EE with respect to $\sqrt{\mbox{DOF}}$ for two-dimensional uniform mesh (UM) and perturbed mesh (PM), left: k=1, middle: k=2, right: k=3.
}\label{EE2d}
\end{figure}

\begin{figure}
\centering
\includegraphics[width=.33\linewidth]{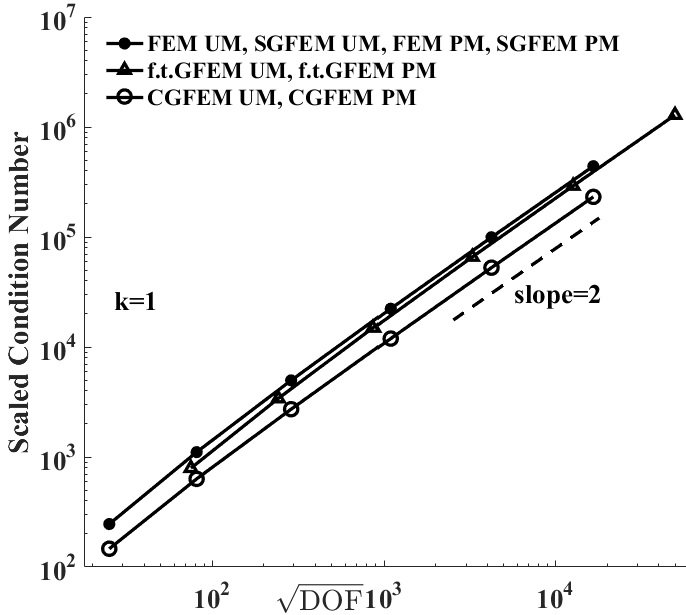}\includegraphics[width=.33\linewidth]{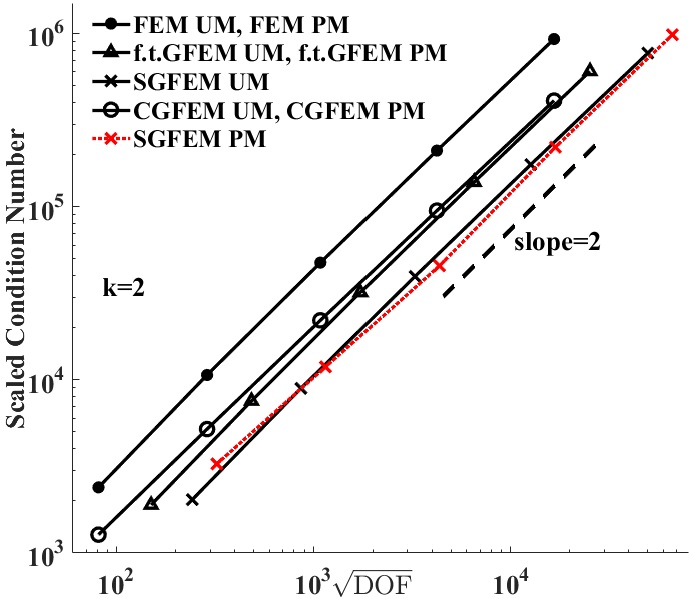}\includegraphics[width=.33\linewidth]{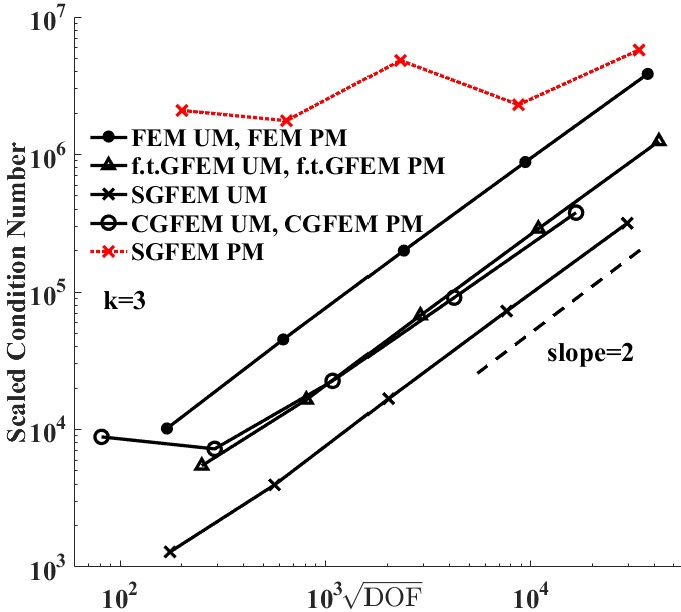}
\caption{ The SCN with respect to $\sqrt{\mbox{DOF}}$ for two-dimensional uniform mesh (UM) and perturbed mesh (PM), left: k=1, middle: k=2, right: k=3.
}\label{SCN2d}
\end{figure}

The EEs and SCNs of FEM, f.t.GFEM, SGFEM, and CGFEM, associated with the $\sqrt{\mbox{DOF}}$, are presented in Fig. \ref{EE2d} and Fig. \ref{SCN2d}, respectively. For the uniform mesh, all four methods yield the optimal convergence order $O(h^k)$ for $k=1,2,3$, where the EEs of CGFEM are smaller than those of the other three methods for $k=1,2$, while the EEs of FEM are minimal for $k=3$. For the perturbed mesh, the EEs of f.t.GFEM, SGFEM, and CGFEM are optimal with the orders $O(h^k)$ for $k=1,2,3$, where the EEs of CGFEM are minimal for all $k=1, 2, 3$. For the perturbed mesh, it is noteworthy that the EEs of FEM are $O(h)$ for all $k=1,2,3$, which are not optimal for $k=2,3$. This coincide with observations in \cite{FrBelIntPU}. It means that the high order FEMs may suffer from mesh distortions. These results show that the CGFEM converges with the optimal order $O(h^k)$ for $k=1,2,3$,
and the associated EEs are smaller than those of the other GFEMs. The EEs of CGFEM also smaller than those of the FEM except in the situation of uniform mesh and $k=3$. Moreover, the CGFEM is robust with the mesh distortions, compared with the FEM.

For the conditioning, it is shown in Fig. \ref{SCN2d} that the four methods are all well conditioned, and the SCNs of CGFEM, f.t.GFEM, and SGFEM increase with an order about $O(h^{-2})$ that is the same as that of standard FEM.
The SCN order of SGFEM for the perturbed mesh and $k=3$ (in Fig. \ref{SCN2d} right) can be shown for larger DOFs.

It is concluded from this set of numerical experiments that (1) the CGFEM converges with the optimal convergence order $O(h^k)$ for $k=1,2,3$; (2) the EEs of CGFEM are minimal compared with the other GFEMs, and are even smaller than those of FEM in most situations; (3) the CGFEM is stable in a sense that its SCNs are of same order as those of FEM; (4) the CGFEM is robust with the mesh distortions, compared with the FEM.

\subsection{2D crack problem} \label{sec6.1}
We next consider a Poisson crack problem. Let $\Omega=[-1,1]^2\setminus [-1,0]\times[0,0]$ be a cracked domain with the crack line $\Gamma_O$ ($x^{[2]}=0, -1\leq x^{[1]}\leq 0$) and the tip $O$ (0,0), as shown in Fig. \ref{cradomain}. We consider the model problem (\ref{craequ}) on $\Omega$ with a manufactured solution
\begin{equation}
u = r^{\frac{1}{2}} \sin (\frac{1}{2}\theta) + r^{\frac{3}{2}} \sin (\frac{3}{2}\theta), \quad -\pi \le \theta \le \pi , \label{MS}
\end{equation}
where $(r,\theta)$ is the polar coordinate with the crack tip $O$ serving as the pole and the opposite direction of crack line as the polar line, i.e., $x^{[2]}=0, x^{[1]}\geq 0$. $f$ and $g$ in (\ref{craequ}) are calculated through the equation (\ref{craequ}) using the exact solution $u$ (\ref{MS}).

We use uniform $n\times n$ meshes on $\overline{\Omega}$ with $n$ odd such that $\Gamma_O$ does not align with boundaries of elements, see Fig. \ref{cradomain} for a display of $\Gamma_O$ and the elements with $n=17$. The mesh nodes are denoted by $\{x_i, i\in I_h\}$, where $I_h$ is the index set, and the standard bi-linear FEM hat-functions $\{N_i, i\in I_h\}$ and patches $\{\omega_i, i\in I_h\}$ are employed to construct the CGFEM and GFEM.

We next present the CGFEM. According to the definition of $I_h^1$ and $I_h^2$ in (\ref{Ih12}), we know in Fig. \ref{cradomain} that the nodes in $I_h^1$ and $I_h^2$ are marked by ``$\times$'' and ``$\circ$'', respectively. The nodes in $I_h^2$ are located in a square $D(O,R)$, as shown in Fig. \ref{cradomain}, and we take $R=\frac{1}{4}$ in the tests. The local enrichment spaces $V_i$ and $V_i$-unisolvent set $X_i$, used to construct the CGFEM, are defined in (\ref{CPenr}) and (\ref{CPUS}), respectively, see Fig. \ref{cradomain}.

For comparison, we also test the standard FEM with its approximation subspace
\[\mbox{span}\{N_i, i\in I_h\}\]
and a conventional GFEM with its approximation subspace as follows:
\begin{equation}
\mbox{span}\{N_i:\,i\in I_h\} + \mbox{span}\{N_iH:\,i\in I_h^1\setminus \Delta\} + \mbox{span}\{N_i S^{\frac{1}{2}}:\,i\in I_h^2\}.\label{CPGFEM}
\end{equation}
The indices in $I_h^1$ and $I_h^2$ in (\ref{CPGFEM}) are the same as above, $S^{\frac{1}{2}}$ is defined in (\ref{CPenr}), $H$ is a Heaviside function defined by
\[H(x)=\left\{\begin{array}{ll}1, & x^{[2]}\geq 0,\\-1,&x^{[2]}<0,\end{array}\right.\]
and $\Delta=\{i\in I_h, x_i\in e_s\,\,\mbox{and}\,\,\,O\in e_s\}$. It is clear that $\Delta$ is the index set of nodes of element that contains the crack tip $O$. The GFEM (\ref{CPGFEM}) is called the conventional GFEM with a geometric enrichment in a sense that the nodes in a fixed domain are enriched by the singular function $S^{\frac{1}{2}}$. The GFEM (\ref{CPGFEM}) yields the optimal convergence order $O(h)$ in the energy norm. We refer to \cite{FrBel,BelGra,LabJul,ZhaBaBan1} for the details of GFEM (\ref{CPGFEM}).

\begin{figure}
\begin{center}
\includegraphics[scale=.35]{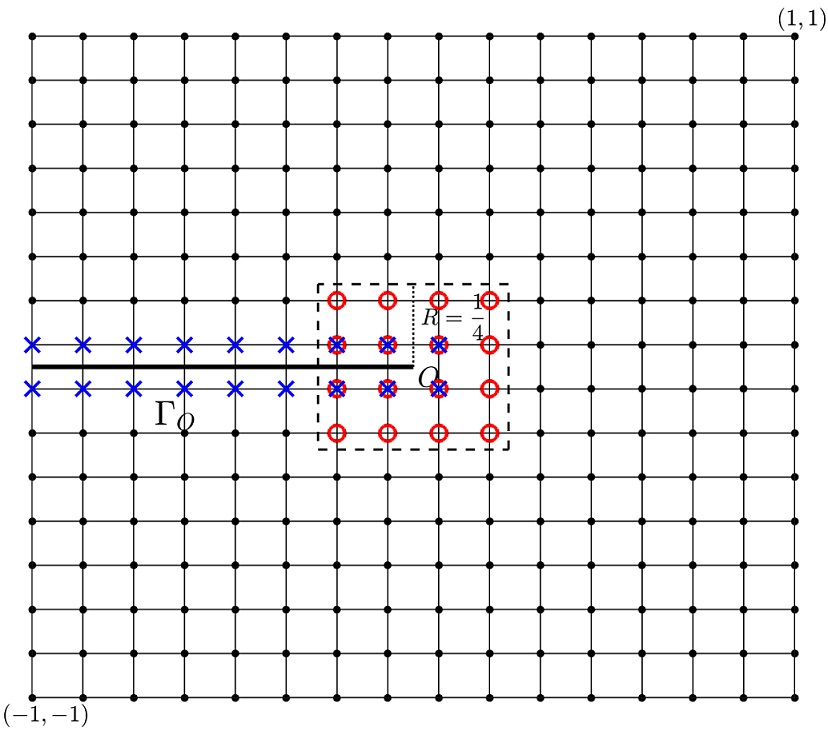}
\caption{{\footnotesize The domain $\Omega$ with a crack line $\Gamma_O$ and a tip $O$. The mesh lines are not aligned with the crack line. For the GFEM, the nodes in $I_h^1$ (``$\times$'') are enriched by $H$, and the nodes in $I_h^2$ (``$\circ$'') are enriched by $S_{\frac{1}{2}}$. In the CGFEM, for the nodes in $I_h^1$ (``$\times$''), the local space $V_i=\mbox{span}\{1, \frac{x^{[1]}-x_i^{[1]}}{h}, \frac{x^{[2]}-x_i^{[2]}}{h},
S_{\frac{1}{2}}, S_{\frac{1}{2}}\frac{x^{[1]}-x_i^{[1]}}{h} \}$; for the nodes in $I_h^2\setminus I_h^1$ (``$\circ$'' without ``$\times$''), $V_h=\mbox{span}\{1, \frac{x^{[1]}-x_i^{[1]}}{h},
\frac{x^{[2]}-x_i^{[2]}}{h}, S_{\frac{1}{2}} \}$; for the other nodes (not ``$\circ$'' or ``$\times$''), $V_i = \mbox{span}\{1, \frac{x^{[1]}-x_i^{[1]}}{h}, \frac{x^{[2]}-x_i^{[2]}}{h}\}$. }}  \label{cradomain}
\end{center}
\end{figure}
The EEs and SCNs of the FEM, GFEM, and CGFEM mentioned above, against the $\sqrt{\mbox{DOF}}$, are presented with $h=\frac{1}{2^{j+1}+1}, j=1,2,\cdots,7$. in Fig. \ref{EESCNcrack}. We observe in Fig. \ref{EESCNcrack} left that both the GFEM and CGFEM yield the optimal order of convergence, i.e., $O(h)$, and the EEs of CGFEM are smaller than those of GFEM. As predicted, the FEM does not produce any convergence order since the solution $u$ is discontinuous cross $\Gamma_O$, and the standard FE functions $N_i$ are continuous. The Fig. \ref{EESCNcrack} right clearly shows that the CGFEM is well conditioned, and its SCNs grow as $O(h^{-2})$, which is same as that of the FEM. On the other hand, the growth of SCNs in the GFEM, approaches $O(h^{-4})$ as $h$ becomes smaller, as observed in \cite{BeMoBu}. Thus the GFEM is not stable.

This set of numerical results demonstrates that the CGFEM is both of optimal convergence and well conditioned for the Poisson crack problems. The extension of CGFEM to the elasticity crack problem will be studied in a future research.
\begin{figure}
\begin{center}
\includegraphics[scale=.50]{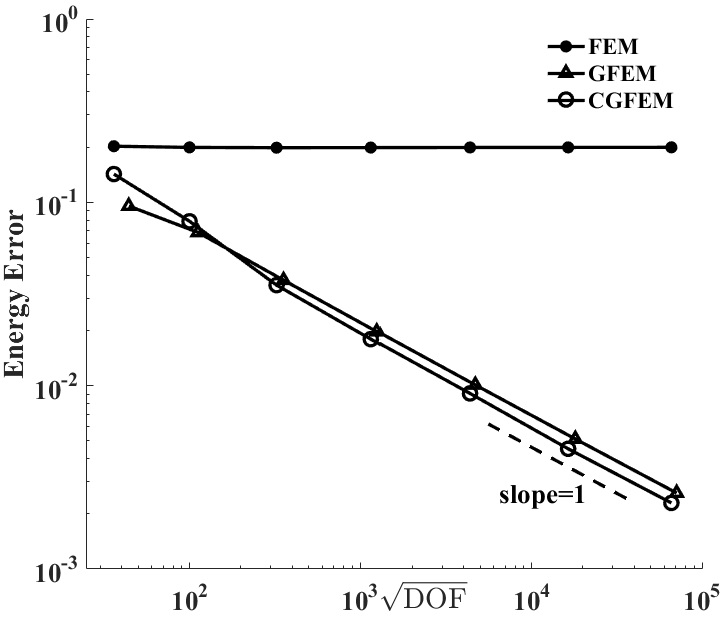}\,\,\,\includegraphics[scale=.50]{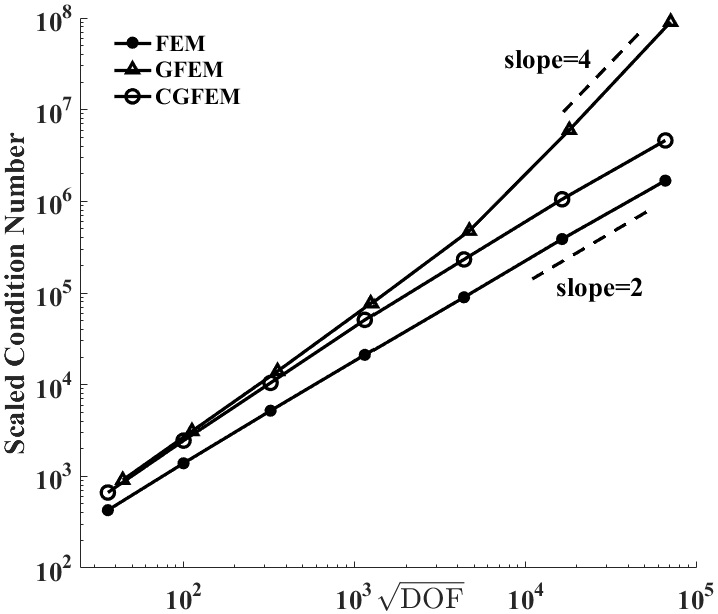}
\caption{{\footnotesize The energy errors (EE) and scaled condition numbers (SCN) of FEM, GFEM and CGFEM for the crack problem. Left: EE; right: SCN.}}  \label{EESCNcrack}
\end{center}
\end{figure}

\section{Concluding Remarks}

We proposed a CGFEM in this paper, and the main idea is to derive a subspace of the approximation space of the conventional GFEM using a local least square procedure and PU techniques. The CGFEM maintains the major advantages of GFEM, such as the approximation properties, simple mesh, and overcome the two main difficulties in GFEM, the conditioning and extra enriched DOFs. Specifically, the number of DOFs of CGFEM is the same as that of FE nodes, and the conditioning is of same order as that of the FEM. The fundamental approximation properties were proven, which are almost same as those of the GFEM. The CGFEM was applied to attain the high order polynomial approximations, and the optimal convergence order $O(h^k)$ under energy norm was proven. The high order CGFEM yields the minimal errors in comparison with the current GFEMs in literatures, such as flat-top GFEM and SGFEM. These errors are even smaller than those of FEM in most situations. Moreover, the CGFEM is robust with mesh distortions, compared with the FEM. We also applied to the CGFEM to a Poisson problem, whose advantages over the conventional GFEM was shown. The extensions of CGFEM to the interface problems, elasticity crack problems, and time-dependent problems will investigated in future studies.

\appendix
\section{Appendix}

We will make comparisons of the proposed method with the conventional GFEM and SGFEM above that employ the FT-PU. To this end, we present an element-wise approach to construct the PU functions $Q_i^\sigma$. For a parameter $0\le \sigma < 0.5$, and a positive integer $l$, we define the FU functions on a 1D reference element $[-1,1]$ as follows:
\begin{eqnarray}\label{PU1D}
Q_{L}^\sigma(\xi)&=&\left\{\begin{array}{ll}1&\xi\in [-1, -1+2\sigma],\\
\left(1-(\frac{\xi+1-2\sigma}{2(1-2\sigma)})^l\right)^l&\xi\in [-1+2\sigma, -1+2(1-\sigma)],\\
0&\xi\in [-1+2(1-\sigma),1],\end{array}\right.\,\,\,\,Q_{R}^\sigma(\xi) = 1- Q_{L}^\sigma(\xi).
\end{eqnarray}
Based on $Q_{L}^\sigma(\xi)$ and $Q_{L}^\sigma(\xi)$, we obtain the PU functions on a 2D reference element $e=[-1,1]^2$:
\begin{equation}\label{PU2D}Q_{L}^\sigma(\xi)\times Q_{L}^\sigma(\eta),\,\,Q_{R}^\sigma(\xi)\times Q_{L}^\sigma(\eta),\,\,Q_{L}^\sigma(\xi)\times Q_{R}^\sigma(\eta),\,\,Q_{R}^\sigma(\xi)\times Q_{R}^\sigma(\eta),\,\,(\xi, \eta)\in [-1,1]^2.
\end{equation}
The FT-PU functions on $e$ are denoted by $Q_J^e$, which are defined in (\ref{PU2D}). Let $e_s$ be any real element 2D, and $F_s$ be the associated affine mapping from the reference element $e$ to $e_s$. Then the FT-PU functions on $e_s$ are derived by
\begin{equation}\label{PU}
Q_J^s(x) = Q_J^e(F_s^{-1}(x)),\,\,\,x\in e_s.
\end{equation}
Assembling these element PU functions $Q_{J}^s$ according to the nodes yields the desired PU functions $Q_i^\sigma, i\in I_h$. It is clear that such PU functions $Q_i^\sigma$ satisfy the condition (\ref{secondPUcond}) since the mesh is assumed to be quasi-uniform. This construction approach can be found out in \cite{ZhaBaBan} in detail.

\end{document}